\documentclass[12pt]{amsart}
\usepackage{amsmath,amsfonts,amsthm,amssymb,mathrsfs}
\usepackage{times,comment}
\usepackage{color,ulem,graphicx}
\theoremstyle{plain}
\newtheorem{theorem}{Theorem}[section]
\newtheorem{assumption}[theorem]{Assumption}
\newtheorem{lemma}[theorem]{Lemma}
\newtheorem{lem}[theorem]{Lemma}
\newtheorem{prop}[theorem]{Proposition}

\theoremstyle{definition}

\newtheorem{defn}[theorem]{Definition}
\newtheorem*{term}{Terminology}
\theoremstyle{remark}
\newtheorem{remark}[theorem]{Remark}
\newtheorem{rmk}[theorem]{Remark}

\newcommand{\vphi}{\varphi}
\newcommand{\veps}{\varepsilon}

\newcommand{\pl}{\partial}

\newcommand{\na}{\nabla}
\newcommand{\lt}{\left}
\newcommand{\rt}{\right}
\newcommand{\rw}{\rightarrow}
\newcommand{\R}{\mathbb{R}}

\renewcommand{\tilde}{\widetilde}

\newcommand{\Lb}{\underline{L}}
\newcommand{\chib}{\underline{\chi}}
\newcommand{\tr}{\operatorname{tr}}

\title{Two rigidity results for surfaces in Schwarzschild spacetimes}
\author[P.-N. Chen]{Po-Ning Chen}
\address{Department of Mathematics, University of California, Riverside, CA, USA}
\email{poningc@ucr.edu}
\thanks{P.-N. Chen is supported by Simons Foundation collaboration grant \#584785}
\author[Y.-K. Wang]{Ye-Kai Wang}
\address{Department of Applied Mathematics, National Yang Ming Chiao Tung University,
Hsinchu, Taiwan}
\address{National Center for Theoretical Sciences\\No. 1 Sec. 4 Roosevelt Rd., National Taiwan
University\\Taipei,106, Taiwan}
\email{ykwang@math.nctu.edu.tw}
\thanks{Y.-K. Wang is supported by Taiwan NSTC grant
109-2628-M-006-001-MY3. We would like to thank professor Mu-Tao Wang for his encouragement.}

\begin{document}
\begin{abstract}
We prove two rigidity results for surfaces lying in the standard null hypersurfaces of Schwarzschild spacetime satisfying certain mean curvature type equations. The first is for the equation $\alpha_H = - d\log |H|$ studied in \cite{WWZ}. The second is for the mean curvature vector of constant norm. The latter is related to the Liouville and Obata Theorem in conformal geometry. \end{abstract}
\maketitle

\section{Introduction}
Alexandrov Theorem \cite{A} is a landmark in differential geometry. It asserts that any closed embedded hypersurface with constant mean curvature in Euclidean space is a round sphere. Later Reilly \cite{Reilly} and Ros \cite{Ros} (see also \cite{MS}) gave different proofs. The methods of proof--Alexandrov's moving plane method, Reilly formula, and Ros inequality--have become indispensable tools in partial differential equations and geometry.
 
In the seminal work \cite{B}, Brendle generalized the Alexandrov theorem by showing that any closed embedded hypersurface of constant mean curvature in a class of warped product manifolds must be umbilical; moreover, for (anti-)de Sitter Schwarzschild and Reissner-Nordstrom manifolds that are of interest in general relativity, he showed that umbilical hypersurfaces are spheres of symmetry. In \cite{WWZ}, Wang, Zhang, and the second-named author proved a spacetime Alexandrov theorem building upon Brendle's result.

Let's recall the setup of \cite{WWZ} which we adopted in the present work. Fix an integer $n \ge 3$. Consider a class of $(n+1)$-dimensional static spherically symmetric spacetimes $(V,\bar{g})$:
\begin{assumption}\label{spacetime_assumption}
\begin{align*}
\bar g = -f^2(r)dt^2 + \frac{1}{f^2(r)}dr^2 + r^2 g_{S^{n-1}}
\end{align*}
where the warping factor $f: I \subset [0,\infty) \rw (0,\infty)$ satisfies the differential inequality
\begin{align}\label{differential inequality}
\frac{f^2-1}{r^2} - \frac{ff'}{r} \leq 0. 
\end{align}
Here $g_{S^{n-1}}$ denotes the metric of the $(n-1)$-dimensional sphere with constant sectional curvature $1$.
\end{assumption}
The equality case of \eqref{differential inequality} produces the Lorentz space forms. Namely, in the case
\[
\frac{f^2-1}{r^2} - \frac{ff'}{r} = 0,
\]
we get $f = 1 - \kappa r^2$ for some constant $\kappa$. $(V,\bar{g})$ then corresponds to the Minkowski ($\kappa=0$), anti-de Sitter ($\kappa <0$), and de Sitter $(\kappa>0)$ spacetimes.

For any codimension 2 submanifold, we consider the connection 1-form of the normal bundle in the mean curvature direction $\alpha_H$
\begin{defn}
When the mean curvature vector is spacelike, we consider the unit spacelike normal $e_n^H = - \frac{H}{|H|}$ and the future-directed unit timelike normal $e_{n+1}^H$ that form an orthonormal basis of the normal bundle. Let 
\begin{align*}
\alpha_H(X) = \langle D_X e_n^H, e_{n+1}^H \rangle \quad \mbox{ for } X\in T\Sigma
\end{align*}
be the connection 1-form of the normal bundle. 
\end{defn}

\begin{term}
A {\it sphere of symmetry} is a sphere $t=t_0,r=r_0$ or its image under a Lorentz transformation when $(V,\bar g)$ is a Lorentz space form (see Subsection \ref{sphere of symmetry}). The null hypersurface emanating from a sphere of symmetry is called a {\it standard null cone}.
\end{term}

The main theorem of \cite{WWZ} is: 
\begin{theorem}\label{main WWZ}
Let $\Sigma$ be a spacelike codimension 2 submanifold in a static spherically symmetric spacetime satisfying Assumption \ref{spacetime_assumption}. Suppose
\begin{align}\label{CNNC}
\alpha_H = - d \log |H|
\end{align}
holds on $\Sigma$. Then $\Sigma$ lies in an outgoing standard null cone provided the past outgoing null hypersurface emanating from $\Sigma$ intersects a time slice at an embedded hypersurface of positive mean curvature. \end{theorem}
We remark that for Minkowski spacetime, the rigidity is proved unconditionally in \cite{HMR} using spinor method.

In this paper, we prove two rigidity results for spacelike codimension 2 submanifolds lying in a standard null cone. \begin{theorem}\label{inf CNNC rigidity}
Let $\Sigma$ be a spacelike hypersurface in the standard null cone of a static spherically symmetric spacetime satisfying Assumption \ref{spacetime_assumption}. Then $\Sigma$ is infinitesimally rigid (see Definition \ref{def: inf CNNC rigid}) for \eqref{CNNC}. 
\end{theorem}

\begin{theorem}\label{constant mean curvature vector norm}
Let $\Sigma$ be a spacelike codimension 2 submanifold lying in the standard null cone of a static spherically symmetric spacetime $(V,\bar{g})$ satisfying Assumption \ref{spacetime_assumption}. Suppose the mean curvature vector of $\Sigma$ has constant norm. Then $\Sigma$ is a sphere of symmetry.\end{theorem}

\begin{figure}
\centering \includegraphics[scale=0.7]{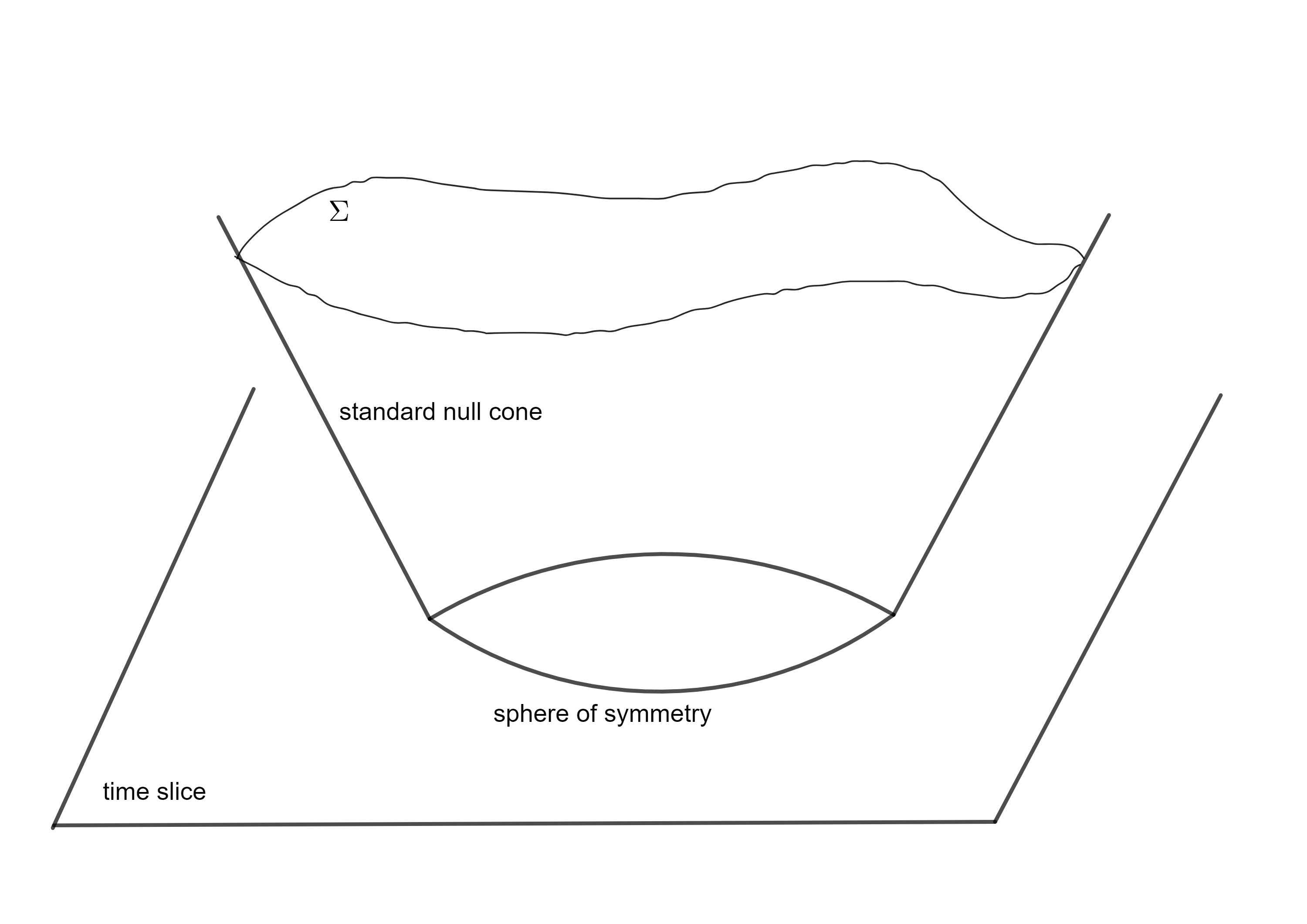}
\end{figure}

In the preliminary Section 2, we discuss the geometry of codimension 2 submanifolds in static spherically symmetric spacetimes. Special care, given in Section 2.3, is needed for the Lorentz space forms because of the additional boost isometries (Lorentz transformations). In Section 3, we prove Theorem \ref{inf CNNC rigidity}. It is an infinitesimal version of Theorem \ref{main WWZ}. While one expects it to hold, the proof here uses only elementary integration by parts instead of the spacetime Minkowski formula and monotonicity formula established in \cite{WWZ}. In Section 4, we begin with a discussion of the constant Gauss curvature equation which turns out to be equivalent to finding surfaces in the standard null cone of Minkowski spacetime with constant mean curvature vector norm. The proof of Theorem \ref{constant mean curvature vector norm} similar to that of Theorem \ref{inf CNNC rigidity} and leads to the Obata Theorem on the uniqueness of constant scalar curvature metric in conformal geometry.

We remark that Theorem \ref{inf CNNC rigidity} and Theorem \ref{constant mean curvature vector norm} hold in static warped product spacetimes. See the end of Section 3 and 4 respectively.
\section{Geometry of codimension 2 submanifolds in static spherically symmetric spacetimes}
\subsection{Static spherically symmetric spacetimes}
 
Note that \eqref{differential inequality} is just the assumption (H4) in \cite{B}. Indeed, for a warped product metric $g=\frac{1}{f^2}dr^2 + r^2 g_{S^2}$, one has $Ric(\nu,\nu) = - (n-1)\frac{ff'}{r}$ and $Ric(e_1,e_1) = (n-2)\frac{1-f^2}{r^2} - \frac{ff'}{r} $ where $\nu$ and $e_1$ are unit normal and unit tangent vector of the sphere of symmetry with area radius $r$. \eqref{differential inequality} thus means that the Ricci curvature is smallest in the radial direction.

The differential inequality \eqref{differential inequality} is related to the null convergence condition in general relativity \cite[page 95]{HE}, which requires that Ricci curvature of the spacetime satisfy $\overline{Ric}(W,W) \ge 0$ for any null vector $W$. 

\begin{lemma} In the following two cases 
\begin{enumerate}
\item $I = [0,r_1) \quad f(0)=1, f'(0)=0$ 
\item $I=(r_0,r_1),\quad \lim_{r \rw r_0^+}f(r)=0$
\end{enumerate}
the null convergence condition is equivalent to the differential inequality \eqref{differential inequality}.
\end{lemma}
\begin{proof}
We consider the inequality $r^{n-1}ff' + r^{n-2}(1-f^2) \ge 0$, which is equivalent to \eqref{differential inequality}. Since we have $r^{n-1}ff' + r^{n-2}(1-f^2) \ge 0$ at $r=0$ in case 1 and as $r \rw r_0^+$, it suffices to check that the quantity has nonnegative derivative.  
 
Let $e_1$ be a unit tangent vector on the sphere of symmetry. Null convergence condition implies that
\[ 0 \leq \overline{Ric} \lt( \frac{1}{f} \frac{\pl}{\pl t} + e_1, \frac{1}{f} \frac{\pl}{\pl t} + e_1 \rt) = (n-3) \frac{ff'}{r}+\frac{1}{2} (f^2)'' + (n-2) \frac{1}{r^2}(1-f^2). \]
On the other hand,
\begin{align*}
&\big( r^{n-1} ff' + r^{n-2} (1-f^2) \Big)' \\
&= r^{n-1} \lt[ (n-3) \frac{ff'}{r}+\frac{1}{2} (f^2)'' + (n-2) \frac{1}{r^2}(1-f^2) \rt] \geq 0.
\end{align*}
This completes the proof. 
\end{proof}
\begin{rmk}
Case 1 includes the Lorentz space forms. Case 2 includes the Schwarzschild spacetime with positive mass $m$ where $f(r) = \sqrt{1 - \frac{2m}{r}}$ and $r_0 = 2m$. 
\end{rmk}

\subsection{Geometry of codimension 2 submanifolds}
Let $(V,\bar g)$ be an $(n+1)$-dimensional time-oriented spacetime with Levi-Civita connection $D$. Let $\Sigma$ be a spacelike codimension 2 submanifold in $V$. We will write $\langle X,Y \rangle = \bar g(X,Y)$ for the metric on $V$. Let $\sigma$ and $H$ be the induced metric and the mean curvature vector of $\Sigma$. The normal bundle of $\Sigma$ is of rank 2 and is spanned by two future-directed null normal vector fields $L$ and $\Lb$, normalized by $\langle L,\Lb \rangle=-2$. Let 
\begin{align*}
\chi_{ab} &= \langle D_a L, \pl_b \rangle, \quad \chib_{ab} = \langle D_a \Lb, \pl_b \rangle \\
\tr\chi &= \sigma^{ab} \chi_{ab}, \quad \tr\chib = \sigma^{ab}\chib_{ab}\\
\zeta_a &= \frac{1}{2}\langle D_a L,\Lb \rangle
\end{align*}
be the null second fundamental forms, null mean curvatures, and the torsion 1-form of $\Sigma$ with respect to $L$ and $\Lb$. If we take $L = e_n^H + e_{n+1}^H$ and $\Lb = -e_n^H + e_{n+1}^H$, then $\zeta = \alpha_H$.

From the rest of the paper, the spacetime $(V,\bar g)$ is static and spherically symmetric as given in Assumption \ref{spacetime_assumption}. The spheres $t=t_0, r= r_0$ are referred to as {\it spheres of symmetry}, and the null hypersurfaces emanating from the spheres of symmetry are referred to as {\it standard null cones}.

It is convenient to work in the Eddington--Finkelstein coordinates. Let $r^* = \int \frac{dr}{f^2}$ be the tortoise coordinate and $v= t+r^*, w = t-r^*$ be the advanced and retarded time. The metric in the Eddington--Finkelstein coordinates \cite[page 153]{HE} becomes $\bar{g} = -f^2 dv dw + r^2 g_{S^{n-1}}.$ ($r$ is implicitly a function of coordinates $v$ and $w$). 
The outgoing standard null cones are defined by $w=w_0$. For a spacelike codimension 2 submanifold $\Sigma$ lying in the standard null cone, we view it as a graph over the sphere of symmetry. Namely, $\Sigma = F(S^{n-1})$ where $F: S^{n-1} \rw V$ 
$$ F(x) = (v(x), w_0,x). $$ 
We abuse the notation to write $r= r \circ F: S^{n-1} \rw \R$ and $v = v\circ F: S^{n-1} \rw \R$.  

The tangent vector of $\Sigma $ is given by $\frac{\pl F}{\pl \theta^a} = \frac{\pl v}{\pl \theta^a} \frac{\pl}{\pl v} + \frac{\pl}{\pl \theta^a}.$ We always choose the null normal vector field $L$ that is tangent to $w=w_0$ to be
\begin{align}
L = \frac{2r}{f^2} \frac{\pl}{\pl v},
\end{align}
which satisfies
\begin{align*}
D_{\frac{\pl F}{\pl\theta^a}} L &= \vphi' \frac{\pl v}{\pl\theta^a} \frac{\pl}{\pl v} + \vphi \left( \frac{\pl v}{\pl\theta^a} D_{\frac{\pl}{\pl v}} \frac{\pl}{\pl v} + D_{\frac{\pl}{\pl\theta^a}} \frac{\pl}{\pl v} \right)\\
&= \vphi' \frac{\pl v}{\pl\theta^a} \frac{\pl}{\pl v} + \vphi \frac{\pl v}{\pl\theta^a} \frac{\pl \log f}{\pl v} \frac{\pl}{\pl v} + \vphi \frac{f^2}{2r} \frac{\pl}{\pl\theta^a}\\
&= \left( \vphi' + \vphi \frac{\pl \log f}{\pl v} \right) \frac{\pl v}{\pl\theta^a} \frac{\pl}{\pl v} + \frac{\pl}{\pl\theta^a}\\
&= \left( \dfrac{2 \frac{\pl r}{\pl v} f^2 - 2r \frac{\pl f^2}{\pl v}}{f^4} + \frac{2r}{f^2} \frac{\pl \log f}{\pl v} \right) \frac{\pl v}{\pl\theta^a} \frac{\pl}{\pl v} + \frac{\pl}{\pl \theta^a} \\
&= \frac{2}{f^2} \frac{\pl r}{\pl v} \frac{\pl v}{\pl\theta^a} \frac{\pl}{\pl v} + \frac{\pl}{\pl \theta^a} \\
&= \frac{\pl F}{\pl\theta^a}.
\end{align*}
In the above computation, the following Christoffel symbols of $\bar g$ are used
\begin{align*}
\bar\Gamma_{av}^b &= \frac{1}{r} \frac{\pl r}{\pl v} \delta_a^b \\
\bar\Gamma_{vv}^v &= \frac{\pl \log f}{\pl v}.
\end{align*}
As a result, $\langle \vec{H},L \rangle = -2$. Let $\Lb$ be the null normal complement to $L$ such that $\langle L, \Lb \rangle = -2.$ It is straightforward to show that
\begin{align}
\Lb= \frac{1}{r} \lt( 2 \frac{\pl}{\pl w} + f^2 \na v - \frac{f^2}{2}|\na v|^2 \frac{\pl}{\pl v} \rt)
\end{align}
where $\na v = \sigma^{ab}\frac{\pl v}{\pl\theta^a} \frac{\pl}{\pl \theta^b}$. We end this subsection with 
\begin{lemma}\label{surface on standard null cone has CNNC}
Spacelike codimension 2 submanifolds lying in a standard null cone satisfy \eqref{CNNC}.\end{lemma}
\begin{proof}
 We have  
\begin{align*}
\vec{H} = \Lb + \psi L \\
\vec{J} = \Lb - \psi L
\end{align*}
with $|\vec{H}|^2 = -4 \psi.$
The connection 1-form thus satisfies
\begin{align*}
(\alpha_H)_a &= \langle D_a \left( -\frac{\vec{H}}{|\vec{H}|} \right), \frac{\vec{J}}{|\vec{H}|} \rangle \\
&= -\frac{1}{|\vec{H}|^2} \langle D_a (\Lb + \psi L), \Lb - \psi L \rangle \\
&= -\pl_a \log |\vec{H}|.
\end{align*}
\end{proof}

\subsection{Lorentz space forms}
Lorentz space forms (Minkowski, de Sitter, anti-de Sitter spacetimes) admit additional isometries compared to generic spherically symmetric static spacetimes. In the proof of our main results, it is essential to include these additional isometries. We review the Lorentz space forms in this subsection.
\subsubsection{Killing vector fields along the surfaces}\label{Killing}
For any spherically symmetric static spacetime, we have
\begin{prop}
Along any spacelike codimension 2 submanifold lying in a standard null cone $w=w_0$ of static spherically symmetric spacetime, we have
\begin{align*}
\langle \frac{\pl}{\pl t}, L \rangle = -r.
\end{align*}
\end{prop}
\begin{proof}
Recall we take $L = \frac{2r}{f^2} \frac{\pl}{\pl v}$. As $\frac{\pl}{\pl v} = \frac{1}{2}(\frac{\pl}{\pl t} + f^2 \frac{\pl}{\pl r})$, the assertion follows.
\end{proof}
Next, we describe the Killing vector fields along the surfaces in the Minkowski spacetime.
\begin{prop}
Let $\mathcal{K}$ denote the boost Killing vector fields $y^0\frac{\pl}{\pl y^i} + y^i\frac{\pl}{\pl y^0}, i=1,\cdots,n$ of the $(n+1)$-dimensional Minkowski spacetime. Then along a spacelike codimension 2 submanifold lying in a standard null cone $w=w_0$, we have
\begin{align*}
\langle \mathcal{K},L \rangle = rw_0\tilde X^i
\end{align*} 
\end{prop}
\begin{proof}
We write $\tilde g = g_{S^{n-1}}$ in the proof. In spherical coordinates,
\begin{align*}
\mathcal{K} = r\tilde X^i \frac{\pl}{\pl t} + t \lt( \tilde X^i \frac{\pl}{\pl r} + \frac{\tilde\na^a \tilde X^i}{r} \frac{\pl}{\pl\theta^a} \rt).
\end{align*}
Recall that $L = \frac{2}{r} \frac{\pl}{\pl v}$ with $\frac{\pl}{\pl v} = \frac{1}{2}\lt( \frac{\pl}{\pl t} + \frac{\pl}{\pl r} \rt)$. As $t = r + w_0$ on $\Sigma$, we get $
\langle \mathcal{K},L \rangle = r w_0 \tilde X^i$.
\end{proof}

The $(n+1)$-dimensional anti-de Sitter spacetime is defined as the quadric
\[ -(y^0)^2 + (y^1)^2 + (y^2)^2 +\cdots  + (y^{n})^2 - (y^{n+1})^2 = -l^2\]
embedded in $\R^{n,2}$ with the metric
\[ - (dy^0)^2 + (dy^1)^2 + (dy^2)^2 + \cdots  + (dy^{n})^2- (dy^{n+1})^2.\]
Anti-de Sitter spacetime is a space form with sectional curvature $-\frac{1}{l^2}$.

The {\it static patch} is given by the coordinate system
\begin{align*}
y^0 &= \sqrt{l^2 + r^2} \sin (t/l)\\
y^i &= r \tilde X^i,\quad i = 1,\cdots n\\
y^{n+1} &= \sqrt{l^2+r^2} \cos (t/l)
\end{align*}
in which the metric is given by
\[ - f^2 dt^2 + \frac{1}{f^2}dr^2 + r^2 g_{S^{n-1}} \]
where $f(r) = \sqrt{1+(r/l)^2}$.

The Killing vector fields $y^0 \frac{\pl}{\pl y^i} + y^i \frac{\pl}{\pl y^0}, y^{n+1} \frac{\pl}{\pl y^i} + y^i \frac{\pl}{\pl y^{n+1}}, i = 1,\cdots,n$ of $\R^{n,2}$
are tangent to the quadric and thus give rise to Killing vector fields of the anti-de Sitter spacetime. 
\begin{lem} Let $\mathcal{K}$ and $\mathcal{K}'$ denote the Killing vector field of anti-de Sitter spacetime given by $y^0\frac{\pl}{\pl y^i} + y^i\frac{\pl}{\pl y^0}$ and $y^{n+1}\frac{\pl}{\pl y^i} + y^i\frac{\pl}{\pl y^{n+1}}$ respectively.
In the static patch, 
\begin{align*}
\mathcal{K} &= \frac{r\tilde X^i}{f} \cos(t/l) \frac{\pl}{\pl t} + lf \tilde X^i \sin(t/l) \frac{\pl}{\pl r} + W\\
\mathcal{K}' &= -\frac{r\tilde X^i}{f} \sin(t/l) \frac{\pl}{\pl t} + lf \tilde X^i \cos(t/l) \frac{\pl}{\pl r} + W'
\end{align*}
where $W$ and $W'$ are tangent to the sphere of symmetry.
\end{lem}
\begin{proof}
We extend the coordinate functions $t$ and $r$ off the quadric by
\begin{align*}
t &= l \tan^{-1} \lt( \frac{y^0}{y^{n+1}} \rt)\\
r &= (y^1)^2 + \cdots (y^n)^2.
\end{align*}
Suppose $y^0\frac{\pl}{\pl y^i} + y^i\frac{\pl}{\pl y^0} = A \frac{\pl}{\pl t} + B \frac{\pl}{\pl r} + W$ where $W$ is tangent to the sphere of symmetry. We compute
\begin{align*}
A = \lt( y^0\frac{\pl}{\pl y^i} + y^i\frac{\pl}{\pl y^0} \rt)t = y^i l \frac{y^{n+1}}{(y^0)^2 + (y^{n+1})^2} = \frac{r\tilde X^i}{f} \cos(t/l),
\end{align*}
and
\begin{align*}
B = \lt(y^0\frac{\pl}{\pl y^i} + y^i\frac{\pl}{\pl y^0} \rt) r = y^0 \frac{y^i}{r} = lf \tilde X^i \sin(t/l).
\end{align*}
The formula for $\mathcal{K}'$ is obtained similarly.
\end{proof}

\begin{rmk}
The Killing vector field $y^0 \frac{\pl}{\pl y^{n+1}} - y^{n+1}\frac{\pl}{\pl y^0}$ gives rise to the time translating Killing vector field $-l \frac{\pl}{\pl t}$ in the static patch. 
\end{rmk}

\begin{prop}
Let $\mathcal{K}$ and $\mathcal{K}'$ denote the Killing vector field of anti-de Sitter spacetime given by $y^0\frac{\pl}{\pl y^i} + y^i\frac{\pl}{\pl y^0}$ and $y^{n+1}\frac{\pl}{\pl y^i} + y^i\frac{\pl}{\pl y^{n+1}}$ respectively. Then along a spacelike codimension 2 submanifold lying in the standard null cone $w = w_0$, we have
\begin{align*}
\langle \mathcal{K},L \rangle &= r \cdot l\sin(w_0/l)\tilde X^i \\
\langle \mathcal{K}',L \rangle &= r \cdot l\cos(w_0/l)\tilde X^i.
\end{align*}
\end{prop}
\begin{proof}
We have the relation \begin{align*}
\frac{\pl}{\pl v} &= \frac{1}{2} \lt( \frac{\pl}{\pl t} + f^2 \frac{\pl}{\pl r}\rt)\\
\frac{\pl}{\pl w} &= \frac{1}{2} \lt( \frac{\pl}{\pl t} - f^2 \frac{\pl}{\pl r} \rt).
\end{align*}
Recall that $L = \frac{2r}{f^2} \frac{\pl}{\pl v} $ and we have
\begin{align*}
\langle \mathcal{K},L \rangle &= \frac{r}{f}\tilde X^i \lt( -r \cos(t/l) + l\sin(t/l)\rt).
\end{align*}
Also, recall that $t = r_* + w_0 = l \tan^{-1}(r/l)+ w_0$ (the constant of integration does not matter), and hence by the addition formula, we obtain
\begin{align*}
-r \cos(t/l) + l\sin(t/l) &= -r \lt( \frac{l}{\sqrt{l^2+r^2}} \cos(w_0/l) - \frac{r}{\sqrt{l^2+r^2}}\sin(w_0/l) \rt) \\
&\quad + l \lt( \frac{r}{\sqrt{l^2+r^2}} \cos(w_0/l) + \frac{l}{\sqrt{l^2+r^2}} \sin(w_0/l) \rt)\\
&= lf \sin(w_0/l).
\end{align*}
The formula for $\langle \mathcal{K},L \rangle$ follows. $\langle \mathcal{K'},L \rangle$ is computed similarly.
\end{proof}

The situation of de Sitter spacetime is similar. We describe the results without proof. The $(n+1)$-dimensional de Sitter spacetime is defined as the quadric
\[ -(y^0)^2 + (y^1)^2 + (y^2)^2 + \cdots + (y^n)^2 + (y^{n+1})^2 = l^2 \]
embedded in $\R^{n+1,1}$ with the metric
\[ -(dy^0)^2 + (dy^1)^2 + (dy^2)^2 + \cdots + (dy^{n+1})^2. \]
De Sitter spacetime is a space form with sectional curvature $\frac{1}{l^2}$.
 
The {\it static patch} is given by the coordinate system
\begin{align*}
y^0 &= \sqrt{l^2-r^2} \sinh (t/l)\\
y^i &= r \tilde X^i, \quad i=1,\cdots,n\\
y^{n+1} &= \sqrt{l^2-r^2} \cosh (t/l)
\end{align*}
in which the metric is given by
\[ - f^2 dt^2 + \frac{1}{f^2}dr^2 + r^2 g_{S^{n-1}}\]
where $f(r) = \sqrt{1 - (r/l)^2}$.
 
The Killing vector fields $y^0 \frac{\pl}{\pl y^i} + y^i \frac{\pl}{\pl y^0}, i = 1,\cdots,n$ of $\R^{n+1,1}$
are tangent to the quadric and thus give rise to Killing vector fields of the de Sitter spacetime.
\begin{lem}
Let $\mathcal{K}$ denote the Killing vector fields of de Sitter spacetime given by $y^0 \frac{\pl}{\pl y^i} + y^i \frac{\pl}{\pl y^0}, i = 1,\cdots,n$. In the static patch,
\begin{align*}
\mathcal{K} = \frac{r\tilde X^i}{f} \cosh(t/l)\frac{\pl}{\pl t} + lf \tilde X^i \sinh(t/l)\frac{\pl}{\pl r} + W 
\end{align*}
where $W$ is tangent to the sphere of symmetry.
\end{lem}
\begin{rmk}
The Killing vector field $y^0 \frac{\pl}{\pl y^{n+1}}+ y^{n+1}\frac{\pl}{\pl y^0}$ gives rise to the time translating Killing vector field $l\frac{\pl}{\pl t}$ in the static patch.
\end{rmk} 
\begin{prop}
Let $\mathcal{K}$ denote the Killing vector fields of de Sitter spacetime given by $y^0 \frac{\pl}{\pl y^i} + y^i \frac{\pl}{\pl y^0}, i = 1,\cdots,n$. Then along a spacelike codimension 2 submanifold lying in the standard null cone $w=w_0$, we have
\begin{align*}
\langle \mathcal{K},L \rangle = r \cdot l \sinh(w_0/l)\tilde X^i.
\end{align*}
\end{prop}
\subsubsection{Sphere of symmetry in Lorentz space forms}\label{sphere of symmetry}
In this subsection, we study the image of a sphere $t=t_0,r=r_0$ under a Lorentz transformation. By a time translation, we assume that it lies in the future null cone of the origin $t=0, r=0$. 
\begin{prop}\label{Lorentz}
Let $\Sigma$ be a sphere $t=r_0, r=r_0$ in the Minkowski spacetime. Then the image of $\Sigma$ under the Lorentz transformation $T(y^0,\cdots,y^n) = (\cosh\beta\, y^0 + \sinh\beta\, y^n, y^1, \cdots, y^{n-1}, \sinh\beta\, y^0 + \cosh\beta\, y^n)$ is 
\begin{align*}
T(\Sigma) = (r,r\tilde X^1, \cdots, r\tilde X^n)
\end{align*}
where $r =\frac{r_0}{\cosh\beta - \sinh\beta \tilde X^n} $. 
\end{prop} 
\begin{proof}
We carry out the computation for $n=3$. $T(\Sigma)$ is given by the embedding
\[ r_0 (\cosh\beta + \sinh\beta \cos\theta, \sin\theta\cos\phi, \sin\theta\sin\phi, \sinh\beta + \cosh\beta\cos\theta). \]
To prove the assertion, we have to solve
\[ \begin{cases}
r_0 \sin\theta = r(\theta')\sin\theta' \\
r_0 (\sinh\beta + \cosh\beta \cos\theta) = r(\theta')\cos\theta'
\end{cases} \]
Taking the sum of squares of the equations, we get \[ r(\theta') = r_0(\sinh\beta \cos\theta + \cosh\beta).\] Taking the quotient of the equations, we get \[ \tan\theta' = \frac{\sin\theta}{\sinh\beta + \cosh\beta\cos\theta}. \]
Then we have $\cos\theta' = \frac{\sinh\beta + \cosh\beta\cos\theta}{\cosh\beta + \sinh\beta\cos\theta}$ and consequently 
\[ r(\theta') = \frac{r_0}{\cosh\beta - \sinh\beta\cos\theta'}. \] 
\end{proof}

The cases of anti-de Sitter and de Sitter spacetimes are identical. Taking the 4-dimensional anti-de Sitter spacetime with sectional curvature $-1$ for example. The null cone at the origin is given by $t = \tan^{-1} r$ in the static patch. We consider a sphere $\Sigma$ on this null cone given by $t = \tan^{-1} r_0, r= r_0$. It corresponds to the sphere with embedding
\begin{align*}
(r_0, r_0\tilde X^1, r_0\tilde X^2, r_0\tilde X^3,1)
\end{align*}
in $\R^{3,2}$. The Lorentz transformation 
\[ T(y^0, \cdots, y^4) = (\cosh\beta\, y^0 + \sinh\beta\, y^3,y^1,y^2,\sin\beta\,y^0 + \cosh\beta\,y^3,y^4) \] thus has the same effect on $\Sigma$ as in the case of Minkowski spacetime. 

We summarize the above discussion into the following theorem.
\begin{defn}\label{modes}
We identify $S^{n-1}$ as the unit sphere in $\R^n$. Let $\tilde X^i, i=1,\cdots, n$, be the restriction of Cartesian coordinate functions of $\R^n$ on $S^{n-1}$. A function $u$ on $S^{n-1}$ is said to be supported in $\ell \le 1$ modes if $u = a_0 + \sum a_i \tilde X^i$ for some constants $a_0, a_1, \cdots, a_n$. 
\end{defn}

\begin{theorem}\label{isometries}
Let $\Sigma$ be a spacelike codimension 2 submanifold in a standard null cone of a Lorentz space form. If the restriction of the radial coordinate $\frac{1}{r}$ on $\Sigma$ is supported in $\ell\le 1$ modes, then $\Sigma$ is a sphere of symmetry.
\end{theorem}
\begin{proof}
Again we work with $n=3$ to illustrate the idea. By a time translation and a rotation, we may assume $\Sigma$ lies in the null cone of the origin and $r^{-1}(x) = a + b\cos\theta$ with $a > |b|$. We then solve $r_0,\beta$ from $a,b$ such that $r^{-1} = r_0 (\cosh\beta + \sinh\beta \cos\theta)$. It follows that $\Sigma$ is the image of a sphere defined by $r=r_0$, $t = r_0$ $(\tan^{-1}r_0, \tanh^{-1}(r_0)$ resp.) under a Lorentz transformation in Minkowski (anti-de Sitter, de Sitter, resp.) spacetime.	
\end{proof}
\section{Infinitesimal Rigidity for \eqref{CNNC}} 

Let $\Sigma$ be a spacelike codimension 2 submanifold lying in a standard null cone given by the embedding $F(\theta^a) = (v(\theta^a), w=w_0, \theta^a)$ with null normals 
\begin{align*}
L &= \frac{2r}{f^2} \frac{\pl}{\pl v} \\
\Lb &= \frac{1}{r} \lt( 2 \frac{\pl}{\pl w} + f^2 \na v - \frac{f^2}{2}|\na v|^2 \frac{\pl}{\pl v} \rt).
\end{align*}
We study variations of $\Sigma$ in the spacetime so that \eqref{CNNC} is preserved infinitesimally. Since \eqref{CNNC} is preserved if $\Sigma$ moves within the standard null cone, it suffices to consider variations of $\Sigma$ in $\underline{\mathcal{C}}(\Sigma)$--the incoming null hypersurface of $\Sigma$. Extend $\Lb$ and $L$ to $\underline{C}(\Sigma)$ such that $D_{\Lb} \Lb=0$ and $\langle L,\Lb \rangle = -2$. Let $F(x,s): \Sigma \times [0,\veps) \rw \underline{\mathcal{C}}(\Sigma)$ be a variation of $\Sigma$. Let $\frac{\pl F}{\pl s}(x,0) = U(x)\Lb$ and $\Lb(s) = \psi(x,s)\Lb, L(s) = \frac{1}{\psi}L$ be the null normals of $F(\Sigma,s).$

It is proved in Proposition 3.4 of \cite{WWZ} that \eqref{CNNC} is equivalent to the existence of null normal vector fields $L,\Lb$ along $\Sigma$ such that $\langle H,L \rangle = const.$ and $\langle D_a \Lb,L \rangle=0$. Suppose the image of $F(x,s)$ satisfies \eqref{CNNC} for $0 \le s \le \epsilon$. Differentiating at $s=0$ motivates the following definition.  
\begin{defn}
$F(x,s)$ is said to preserve \eqref{CNNC} infinitesimally if \begin{align}\label{infinitesimal CNNC variation}
\begin{split}
\frac{\pl}{\pl s}& \Big|_{s=0} \pl_a \tr\chi =0 \\
\frac{\pl}{\pl s}& \Big|_{s=0} \langle D_a \Lb(s),L(s) \rangle =0.
\end{split}
\end{align}
\end{defn}
Isometry of the spacetime gives rise to solutions to \eqref{infinitesimal CNNC variation}. For any constant $c$, $U=cr$ is a always solution of \eqref{infinitesimal CNNC variation}. Such a solution corresponds to $\Sigma$ being moved by the time-translation Killing vector field $\frac{\pl}{\pl t}$. Moreover, when $(V,\bar{g})$ is a space form, there are additional solutions where the standard null cone is moved by an isometry of $(V,\bar{g})$. In case there are no further solutions to \eqref{infinitesimal CNNC variation}, a codimension 2 submanifold is said to be {\it infinitesimally rigid for \eqref{CNNC}}. 
\begin{defn}\label{def: inf CNNC rigid}
A spacelike codimension 2 submanifold lying in a standard null cone is said to be {\it infinitesimally rigid for \eqref{CNNC}} if all the solutions $U$ of \eqref{infinitesimal CNNC variation} are a constant multiple of $r$ or $U$ is equal to $r$ times a function on $S^{n-1}$ supported in the $\ell \le 1$ modes when $(V,\bar{g})$ is a space form.
\end{defn}

We need the following well-known theorem to prove the infinitesimally rigidity in space form. We include the proof for completeness. 
\begin{theorem}\label{well-known}
Let $(M,g)$ be a complete connected closed Riemannian manifold of dimension $n \ge 2$. Suppose $M$ is an Einstein manifold when $n \ge 3$ or has constant Gauss curvature when $n=2$. If there is a non-constant solution to
\begin{align}\label{concircular}
\na_i\na_j u = \frac{\Delta u}{n} g_{ij},
\end{align}
then $M$ is isometric to a round sphere $S^n(c) \subset \R^{n+1}$. Moreover, 
\[ u(x) = c_1 + c_2 x \cdot a \]
for some constants $c_1,c_2$ and point $a\in S^n(c)$. 
\end{theorem}
\begin{proof}
Taking the divergence of \eqref{concircular}, we see that $\Delta u - \frac{R}{n-1} u$ is a constant function where $R$ is the (constant) scalar curvature. Since $M$ is closed, by Theorem 2 of \cite{Tashiro}, $M$ is isometric to a round sphere. Note that $\eqref{concircular}$ implies that the level sets of $u$ are umbilical and hence are geodesic spheres on a round sphere. The second assertion then follows.  
\end{proof}

We are ready to prove Theorem \ref{inf CNNC rigidity}.

\begin{theorem}[Theorem \ref{inf CNNC rigidity}]
Let $\Sigma$ be a spacelike codimension 2 submanifold lying in a standard null cone of a static spherically symmetric spacetime satisfying Assumption \ref{spacetime_assumption}. Then $\Sigma$ is infinitesimally rigid for \eqref{CNNC}. 
\end{theorem}

\begin{proof}
We write (\ref{infinitesimal CNNC variation}) into a differential equation of $U$. The variation of $L$ is given by
\begin{align}
\langle D_{\pl_s} L(s), \pl_a \rangle \big|_{s=0} &= -\langle L, D_a (u\Lb) \rangle = 2 \na_a U, \\
\langle D_{\pl_s} L(s),\Lb(s) \rangle \big|_{s=0} &= -\langle L, \psi' \Lb \rangle = 2 \psi',
\end{align}
where $\psi' = \frac{\pl\psi}{\pl s}(x,0).$
Hence, we have
\begin{align}
\frac{\pl}{\pl s} &\Big|_{s=0} \pl_a \tr\chi =
 \pl_a \lt( -U \tr\chib + 2 \Delta U - 2 (n-1)\psi' + U \sigma^{bc} \bar{R}(\Lb, \pl_b, \pl_c, L) \rt)\\
\frac{\pl}{\pl s}& \Big|_{s=0} \langle D_a \Lb(s),L(s) \rangle = 2 \chib_{ab} \na^b U + U \bar{R}(\Lb, \pl_a,L,\Lb) - 2 \na_a \psi'
\end{align}
where the first equation follows from
\begin{align*}
&\frac{\pl}{\pl s} \Big|_{s=0} \tr\chi \\
&=\frac{\pl}{\pl s} \Big|_{s=0} \sigma^{bc} \langle D_b L, \pl_c \rangle)\\
&= -2U\chib^{bc}\sigma_{bc} + \sigma^{bc} \lt( \bar R(U\Lb,\pl_b,\pl_c,L) + \langle D_b (-2\psi' L), \pl_c \rangle + \langle D_b L, D_c (U\Lb) \rangle \rt)\\
&= -2U\tr\chib + U \sigma^{bc} \bar R(\Lb,\pl_b,\pl_c,L) - 2(n-1)\psi' + U \tr\chib.
\end{align*}
By Proposition \ref{curvature}, we have
\begin{align}\label{ambient curvature}
\begin{split}
\sigma^{bc} \bar{R}(\Lb,\pl_b,\pl_c,L) &= -2(n-1)\frac{ff'}{r} \\
\bar{R}(\Lb,\pl_a,L,\Lb) &= \frac{4}{r} \frac{\pl r}{\pl\theta^a} \lt( -(ff')' + \frac{ff'}{r} \rt).
\end{split}
\end{align}
Equation (\ref{infinitesimal CNNC variation}) implies
\begin{align}\label{scalar infinitesimal CNNC equation}
- \na_a ( U\tr\chib) + &2 \na_a \Delta U -2(n-1) \na_a \lt( \frac{ff'}{r} U \rt) \notag\\
&- (n-1) \chib_{ab}\na^b U - 2(n-1)U \frac{\na_a r}{r} \lt( (ff')' - \frac{ff'}{r}\rt) =0. 
\end{align}
Let $U = u\cdot r$ in (\ref{scalar infinitesimal CNNC equation}) and we get  
\begin{align}\label{scalar inf CNNC eq'}
-\na_a (ur \tr\chib) - &2(n-1) \na_a \lt( \frac{ff'}{r} ur \rt) + 2 \na_a (\Delta(ur)) \notag\\
&+ 2(n-1) u \na_a r \lt( (ff')' - \frac{ff'}{r} \rt) - (n-1) \chib_{ab} \na^b(ur) =0
\end{align} On $\Sigma$, we have 
\begin{align}
\chib_{ab} &= -\frac{1}{r^2} \lt( (f^2+|\na r|^2) \sigma_{ab} - 2r \na_a\na_b r \rt), \label{chibar of surfaces in standard null cone}\\
\tr\chib &= -\frac{1}{r^2} \lt( (n-1) \lt( f^2+|\na r|^2 \rt) - 2 r\Delta r \rt). \label{trace chibar of surfaces in standard null cone}
\end{align} 
Combining (\ref{scalar inf CNNC eq'}), (\ref{chibar of surfaces in standard null cone}) and (\ref{trace chibar of surfaces in standard null cone}) and the fact that $u=1$ is a solution of  (\ref{scalar inf CNNC eq'}), we have 
\begin{align*}
0 &= \frac{n-1}{r} (f^2+|\na r|^2)\na_a u - 2(n-1)ff' \na_a u + 2 \na_a (\Delta u \cdot r + 2 \na_b u \na^b r) \\
&\quad + \frac{n-1}{r} \lt( (f^2+|\na r|^2) \sigma_{ab} - 2r \na_a\na_b r \rt) \na^b u \\
&= 2(n-1) \lt( \frac{f^2+|\na r|^2}{r} -ff' \rt)\na_a u + 2 \na_a (\Delta u \cdot r) + 4 \na_a\na_b u \na^b r \\
&\quad + 4 \na^b u \na_a\na_b r - 2(n-1)\na_a\na_b r \na^b u.
\end{align*} 
Note that the induced metric of $\Sigma$ is conformal to the standard metric $\tilde{\sigma}$ of $S^{n-1}$: $\sigma = r^2 \tilde{\sigma}$. We rewrite the equation in terms of $\tilde{\sigma}$:
\begin{align}
0 = (n-1)\lt( \frac{f^2}{r} - ff' \rt)\tilde{\na}_a u + \tilde{\na}_a \lt( \frac{\tilde{\Delta} u}{r} \rt) +(n-1) \tilde{\na}_a \tilde{\na}_b u \frac{\tilde{\na}^b r}{r^2} \notag
\end{align}
We multiply the equation by $r^{n-1} \tilde{\na}^a u$ and integrate by parts to get 
\begin{align}\label{Ricci identity1}
\begin{split}
0&= \int_{S^{n-1}} \Big[ (n-1)r^{n-2}(f^2-rff') |\tilde{\na} u|^2 \\
&\qquad \qquad+ r^{n-2} \tilde{\na}_a \tilde{\Delta} u \tilde{\na}^a u - \frac{1}{n-2} \tilde{\Delta}u \tilde{\na}_a (r^{n-2}) \tilde{\na}^a u \\
&\qquad\qquad + \frac{n-1}{n-2} \tilde{\na}_a \tilde{\na}_b u \tilde{\na}^b (r^{n-2}) \tilde{\na}^a u \Big] \\
&= \int_{S^{n-1}} \Bigg[ (n-1)r^{n-2} \lt( f^2-1 -rff' \rt) |\tilde{\na} u|^2 \\
&\qquad\qquad -\frac{n-1}{n-2} r^{n-2} \lt( |\tilde{\na}^2 u|^2 - \frac{1}{n-1} (\tilde{\Delta} u)^2 \rt) \Bigg].
\end{split} 
\end{align}
In the second equality, we make use of the Ricci identity with $Ric(g_{S^{n-1}}) = (n-2) g_{S^{n-1}}$. Here the integral is taken with respect to the volume form of the standard metric $g_{S^{n-1}}$.   
Note that \[ |\tilde{\na}^2 u|^2 - \frac{1}{n-1} (\tilde{\Delta} u)^2 = \lt| \tilde{\na}^2 u - \frac{1}{n-1} (\tilde{\Delta} u) \tilde{\sigma} \rt|^2. \] By \eqref{differential inequality}, either $u = const.$ or the equality of \eqref{differential inequality} holds and $(V,\bar{g})$ is a space form. In the latter case, $u$ is supported in $\ell\le 1$ modes by Theorem \ref{well-known}.
\end{proof}

Let $(N,g_N)$ be an $(n-1)$-dimensional Einstein manifold with $Ric(g_N) = (n-2)g_N$. The above infinitesimal rigidity holds in the $(n+1)$-dimensional static warped product spacetimes
\begin{align*}
\bar g = -\frac{1}{f^2(r)}dt^2 + f^2(r) dr^2 + r^2 g_N
\end{align*}
where the warping factor $f$ satisfies \eqref{differential inequality}. To see this, observe that the curvature of $N$ is only used in \eqref{ambient curvature}.  

\section{Codimension 2 submanifolds with constant norm mean curvature vector}
We first review a fundamental result in 2-dimensional conformal geometry. The stereographic projection identifies $S^2 \subset \R^3$ and $\mathbb{C}$ with
\begin{align*}
z = \frac{x_1+ix_2}{1-x_3}.
\end{align*}
It is well-known that all conformal transformations on $S^2$ arise as fractional linear transformations 
\begin{align*}
w= \frac{az+b}{cz+d}, ad-bc \neq 0
\end{align*}
on $\mathbb{C} \cup \infty$.

Under a conformal transformation, the standard metric on $S^2$, $\frac{4}{(1 +|z|^2)^2} |dz|^2$ in complex coordinate, is changed by \begin{align*}
\frac{4}{(1+|w|^2)^2}|dw|^2 = \frac{4}{(|cz+d|^2+|az+b|^2)^2}|dz|^2 = r^2 \frac{4}{(1+|z|^2)^2}|dz|^2
\end{align*}
where the conformal factor $r$ satisfies
\begin{align*}
\frac{1}{r} &= \frac{|cz+d|^2+|az+b|^2}{1+|z|^2}\\
&=\frac{(|a|^2+|c|^2)|z|^2 + |b|^2+ |d|^2 + \mbox{Re}\lt( z(a\bar{b}+ c\bar{d}) \rt)}{1+|z|^2} \\
&= (|a|^2+|b|^2+|c|^2+|d|^2) + \mbox{Re}(a\bar{b}+ c\bar{d}) x_1 \\
&\quad- \mbox{Im}(a\bar{b}+ c\bar{d}) x_2 +(|a^2|+ |c|^2 -|b|^2 -|d|^2) x_3.
\end{align*}

Let $\tilde\sigma$ denote the standard metric on $S^2$. Denote the covariant derivative and Laplacian with respect to $\tilde\sigma$ by $\tilde\na$ and $\tilde\Delta$. For a metric $e^{2w} \tilde\sigma$ to have constant Gauss curvature $E$, the conformal factor satisfies the equation
\begin{align}
\tilde{\Delta} w + E e^{2w}=1 \label{PSG}.
\end{align}
Let $r=e^w$, $u = \frac{1}{r}$ and (\ref{PSG}) becomes 
\begin{align}
u^2 + u \tilde{\Delta} u - |\tilde{\na} u|^2 = E. \label{PSG''}
\end{align}
The above discussion thus gives a function-theoretic proof of the following uniqueness result. Indeed, there exists a diffeomorphism $\vphi: S^2 \rw S^2$ such that $\vphi^*\tilde\sigma = E^{-1} e^{2w}\tilde\sigma$. Since $\vphi$ is a conformal map, the assertion follows.  
\begin{theorem}\label{Liouville}
All the positive solutions of (\ref{PSG''}) are supported in the $\ell\le 1$ modes.
\end{theorem} 

We now present an analytical proof. The equation is equivalent to
\begin{align*}
\tilde{\Delta} \left( u^2 + u \tilde{\Delta} u - |\tilde{\na} u|^2 \right) =0.
\end{align*}
By the Bochner formula, we have
\begin{align*}
2u \tilde{\Delta} u + (\tilde{\Delta} u)^2 + u \tilde{\Delta}^2 u - 2|\tilde{\na}^2 u|^2 =0,
\end{align*}
which is equivalent to 
\begin{align*}
u \cdot \tilde{\Delta}(\tilde{\Delta} + 2)u - (u_{11} - u_{22})^2 =0.
\end{align*}
The maximum principle implies that $(\tilde{\Delta} + 2)u$ is a constant function.

\begin{remark}
It is mentioned in \cite[page 15]{Chang} that Struwe and Uhlenbeck have an analytic proof.
\end{remark}
The constant Gauss curvature equation \eqref{PSG} has an interpretation in Minkowski spacetime geometry. Let $\Sigma \subset \mathbb{R}^{3,1}$ be a spacelike 2-surface lying in the outgoing null cone of the origin given by the embedding $F:S^2 \rw \mathbb{R}^{3,1}, F(\theta,\phi) = (r(\theta,\phi), r(\theta,\phi),\theta,\phi).$ The induced metric of $\Sigma$ is $r^2 \tilde{\sigma}$. By the Gauss equation, the square norm of the mean curvature vector of $\Sigma$ equals its Gauss curvature. By Proposition \ref{Lorentz}, Theorem \ref{Liouville} says surfaces with constant norm mean curvature vector must come from spheres of symmetry by isometries. 

This can be generalized to $(n+1)$-dimensional static spherically symmetric spacetimes.
\begin{theorem}[Theorem \ref{constant mean curvature vector norm}]
Let $\Sigma$ be a spacelike codimension 2 submanifold lying in the standard null cone of a static spherically symmetric spacetime $(V,\bar{g})$ satisfying Assumption \ref{spacetime_assumption}. Suppose the mean curvature vector of $\Sigma$ has constant norm. Then $\Sigma$ is a sphere of symmetry.\end{theorem}

\begin{proof}
With our choice of $L$ and $\Lb$, $\tr\chi = n-1$. Since $|\vec{H}|^2 = \tr\chi \tr\chib$, the equation we want to investigate is
\begin{align}
E = \frac{1}{r^2} \Big( (n-1)(f^2+|\na r|^2) - 2r \Delta r \Big)
\end{align}
for some constant $E$. As in the proof of Theorem \ref{inf CNNC rigidity}, we rewrite the equation with respect to the standard metric $\tilde{\sigma}$ on $S^{n-1}$ and let $u=\frac{1}{r}$. We obtain
\begin{align*}
E &= \frac{1}{r^2} \Big( (n-1)\lt( f^2 + \frac{|\tilde{\na} r|^2}{r^2} \rt) -2r \lt( \frac{\tilde{\Delta} r}{r^2} + (n-3) \frac{|\tilde{\na} r|^2}{r^3} \rt) \Big) \\
&=u^2 \lt( (n-1)f^2 - (n-1) \frac{|\tilde{\na} u|^2}{u^2} + \frac{2}{u} \tilde{\Delta}u \rt).
\end{align*}
Taking Laplacian, we get
\begin{align}\label{Ricci identity2}\begin{split}
0 &= \tilde{\Delta} \big( (n-1)f^2u^2 \big) - 2(n-1) \lt( |\tilde{\na}^2 u|^2 - \frac{1}{n-1}(\tilde{\Delta} u)^2 \rt) \\
&\quad -(2n-6) \tilde{\na} \tilde{\Delta} u \cdot \tilde{\na} u - 2(n-1)(n-2) |\tilde{\na} u|^2 + 2u \tilde{\Delta}^2 u. \end{split}
\end{align}
Here we use the Ricci identity with $Ric(g_{S^{n-1}}) = g_{S^{n-1}}$. We multiply $u^{2-n}$ and integrate by parts on $(S^{n-1},g_{S^{n-1}})$ to get 
\begin{align*}
0 &= \int_{S^{n-1}} \Bigg[ \frac{(n-1)(n-2)}{u^{n-1}} \tilde{\na}_a \big( f^2u^2 \big) \tilde{\na}^a u \\
&\qquad\qquad - \frac{2(n-1)}{u^{n-2}} \lt( |\tilde{\na}^2 u|^2 - \frac{1}{n-1} (\tilde{\Delta} u)^2 \rt) \\
&\qquad\qquad - \frac{2(n-1)(n-2)}{u^{n-2}} |\tilde{\na} u|^2 \Bigg] \\
&= \int_{S^{n-1}} \Bigg[ \frac{(n-1)(n-2)}{u^{n-1}} \lt( -(f^2)' + \frac{2}{r}(f^2-1) \rt) |\tilde{\na} u|^2 \\
&\qquad\qquad - \frac{2(n-1)}{ u^{n-2}} \lt| \tilde{\na}^2 u - \frac{1}{n-1} (\tilde{\Delta} u) \tilde{\sigma} \rt|^2 \Bigg]  
\end{align*}
Hence $u = const.$ is the only solution unless $\frac{f^2-1}{r^2} - \frac{ff'}{r} =0$. In the latter case, $(V,\bar g)$ is a Lorentz space form, and $u$ is supported in $\ell \le 1$ modes (see Definition \ref{modes}) by Theorem \ref{well-known}. The assertion then follows from Theorem \ref{isometries}.
\end{proof}

The above argument proves the uniqueness of constant scalar curvature metrics in conformal geometry. 
\begin{theorem}
Let $n \ge 2$. Suppose $(\Sigma^n, \sigma)$ is a closed manifold with constant Ricci curvature, $Ric(\sigma) = c\sigma$. Let $\bar{\sigma} = r^2 \sigma$ be a conformal metric with constant scalar curvature, where $r$ is a smooth positive function. Then $r$ must be constant unless $(\Sigma,\sigma)$ is isometric to the standard sphere $(S^n, \sigma_c)$ and 
\[ r(x) = (c_1 + c_2 x \cdot a)^{-1}\]
for some constants $c_1, c_2$ and point $a \in S^n$.
\end{theorem}

\begin{rmk}
The case $n=2$ is just Theorem \ref{Liouville}. The case $n\ge 3$ is known as Obata's Theorem \cite[Theorem 6.2]{Obata}, which Obata proved by studying the traceless Ricci tensor. The proof given below is not new but rather than a special case of Veron-Veron \cite[Theorem 6.1]{BVV}. We are indebted to Professor Xiaodong Wang for bringing \cite{BVV} to our attention. See also Section 2 of \cite{W} for another proof.
\end{rmk}
\begin{proof}
Let $u=\frac{1}{r}$. The scalar curvature under conformal transformation satisfies \cite[(D.9), page 446]{Wald}
\[\bar{R} = u^2 \lt( nc + 2(n-1) \lt( \frac{\Delta u}{u} - \frac{|\na u|^2}{u^2}\rt) - (n-1)(n-2) \frac{|\na u|^2}{u^2} \rt). \]
Simplify the formula, and we get
\[\frac{\bar{R}}{n-1} = \frac{nc}{n-1}u^2 + 2u \Delta u -n|\na u|^2. \]
By the assumption, $\bar{R} = constant$. Taking Laplacian on both sides and using the Bochner formula, we get
\begin{align*}
0 &= \frac{nc}{n-1}(u\Delta u + |\na u|^2) + (\Delta u)^2 + 2 \na u \cdot \na\Delta u + u \Delta^2 u \\
&\quad - n(|\na^2 u|^2 + \na u \cdot \na\Delta u + c |\na u|^2) \\
&= -n (|\na^2 u|^2 - \frac{1}{n}(\Delta u)^2) + (2-n)\na u \cdot \na\Delta u + u \Delta^2 u \\
&\quad + \frac{nc}{n-1}(u\Delta u + (2-n)|\na u|^2).
\end{align*}
Multiplying by $u^{1-n}$ and integrating by parts, we obtain
\[ 0 = \int_\Sigma - \dfrac{n}{u^{n-1}} \lt| \na^2 u - \frac{1}{n}(\Delta u) \sigma \rt|^2 \]
and the assertion follows from Theorem \ref{well-known}.
\end{proof}

Note that the characterization of constant mean curvature norm submanifold also holds in the static warped product spacetimes described in the end of last section, since the argument only uses the Ricci curvature.

\appendix
\section{Curvature of static spherically symmetric spacetimes}
The purpose of this appendix is to compute the (spacetime) curvature of the warped product $(-\infty,\infty) \times (r_1,r_2)\times N$ with metric
\begin{align}\label{warped product}
\bar g = -f^2(r)dt^2 + \frac{1}{f^2(r)}dr^2 + r^2 \tilde g
\end{align}
where the base $(N,\tilde g)$ is an $(n-1)$-dimensional Riemannian manifold. Our convention of curvature is
\begin{align*}
\bar R(X,Y)Z &= D_X D_Y Z - D_Y D_X Z - D_{[X,Y]}Z\\
\bar R(X,Y,Z,W) &=  \bar g (R(X,Y)W,Z ),
\end{align*} 
and in local coordinates,
\begin{align*}
\bar R \lt( \frac{\pl}{\pl x^\alpha}, \frac{\pl}{\pl x^\beta} \rt) \frac{\pl}{\pl \gamma} &= \bar R_{\alpha\beta\;\;\gamma}^{\;\;\;\;\;\delta} \frac{\pl}{\pl x^\delta}\\
\bar R_{\alpha\beta\epsilon\gamma} &= \bar g_{\epsilon\delta} \bar R_{\alpha\beta\;\;\gamma}^{\;\;\;\;\;\delta} = \bar R \lt( \frac{\pl}{\pl x^\alpha},\frac{\pl}{\pl x^\beta},\frac{\pl}{\pl x^\epsilon},\frac{\pl}{\pl x^\gamma} \rt). 
\end{align*}

Let $\theta^a, a = 1,\cdots, n-1$ be a local coordinate system of $N$. The nonzero Christoffel symbols are given by
\begin{align*}
\Gamma_{tt}^r = -\frac{f'}{f}, \Gamma_{tr}^t = \frac{f'}{f}, \Gamma_{rr}^r = -\frac{f'}{f}, \Gamma_{ar}^b = \frac{1}{r}\delta_a^b, \Gamma_{ab}^r = -f^2 r \tilde g_{ab}, \Gamma_{ab}^c = \tilde\Gamma_{ab}^c.
\end{align*}
The nonzero curvature components \[ \bar R_{\alpha\beta\;\;\gamma}^{\;\;\;\;\;\delta} = \frac{\pl}{\pl x^\alpha} \Gamma_{\beta\gamma}^\delta - \frac{\pl}{\pl x^\beta} \Gamma_{\alpha\gamma}^\delta + \Gamma_{\alpha\epsilon}^\delta\Gamma_{\beta\gamma}^\epsilon - \Gamma_{\beta\epsilon}^\delta\Gamma_{\alpha\gamma}^\epsilon \] are
\begin{align*}
\bar R_{trrt} &= -ff'' - (f')^2\\
\bar R_{tabt} &= -rf^3f' \tilde g_{ab}\\
\bar R_{rabr} &= \frac{f'}{f}r^2 \tilde g_{ab}\\
\bar R_{abdc} &= r^2 \tilde R_{abdc} + r f^2 (\tilde g_{ac}\tilde g_{bd} - \tilde g_{ad}\tilde g_{bc}).	 
\end{align*}
In Eddinton--Finkelstein coordinates, recalling $\frac{\pl}{\pl v} = \frac{1}{2}(\frac{\pl}{\pl t} + f^2\frac{\pl}{\pl r})$ and $\frac{\pl}{\pl w} = \frac{1}{2}(\frac{\pl}{\pl t} - f^2\frac{\pl}{\pl r})$, the above translates into
\begin{align*}
\bar R \lt( \frac{\pl}{\pl w},\frac{\pl}{\pl v}, \frac{\pl}{\pl w},\frac{\pl}{\pl v} \rt) &= -\frac{f^4}{4} \lt(ff''+(f')^2 \rt)\\
\bar R \lt( \frac{\pl}{\pl w},\frac{\pl}{\pl\theta^a},\frac{\pl}{\pl\theta^b},\frac{\pl}{\pl v} \rt) &= -\frac{1}{2} r f^3f'\tilde g_{ab}\\
\bar R \lt( \frac{\pl}{\pl w},\frac{\pl}{\pl\theta^a},\frac{\pl}{\pl\theta^b},\frac{\pl}{\pl w} \rt) &= \bar R \lt(\frac{\pl}{\pl v},\frac{\pl}{\pl\theta^a},\frac{\pl}{\pl\theta^b},\frac{\pl}{\pl v} \rt) =0. 
\end{align*}

\begin{prop}\label{curvature}
Let $\Sigma$ be a spacelike codimension 2 submanifold lying the standard null cone in an $(n+1)$-dimensional spacetime with metric \eqref{warped product}. Let $\sigma$ denote the induced metric of $\Sigma$. Then \begin{align*}
\sigma^{ab} \bar R(\Lb,\pl_a,\pl_b,L) &= -2(n-1)\frac{ff'}{r} \\
\bar R(\Lb,\pl_a,L,\Lb) &= \frac{4}{r} \frac{\pl r}{\pl\theta^a} \lt( -(ff')' + \frac{ff'}{r} \rt).
\end{align*}  
\end{prop}
\begin{proof}
Recall that we take
\begin{align*}
L &= \frac{2r}{f^2} \frac{\pl}{\pl v} \\
\Lb &= \frac{1}{r} \lt( 2 \frac{\pl}{\pl w} + f^2 \na v - \frac{f^2}{2}|\na v|^2 \frac{\pl}{\pl v} \rt)
\end{align*}
along $\Sigma$ where $\na v = \sigma^{ab} \frac{\pl v}{\pl\theta^a} \frac{\pl}{\pl\theta^b}$. Moreover, $\pl_a = \frac{\pl v}{\pl\theta^a}\frac{\pl}{\pl v} + \frac{\pl}{\pl\theta^a}$. We compute
\begin{align*}
\bar R (\Lb,\pl_a,\pl_b,L) &= \frac{2}{f^2} \bar R \lt( 2\frac{\pl}{\pl w},\frac{\pl}{\pl\theta^a},\frac{\pl}{\pl\theta^b},\frac{\pl}{\pl v} \rt) = -2 r ff' \tilde g_{ab}.
\end{align*}
Since $\sigma_{ab} = r^2 \tilde g_{ab}$, the first formula follows. For the second formula, we note that on $\Sigma,$ $v = 2r^* + const.$ and hence $\frac{\pl v}{\pl\theta^a} = \frac{2}{f^2}\frac{\pl r}{\pl\theta^a}$. Hence,
\begin{align*}
&\bar R(\Lb,\pl_a,L,\Lb) \\
&= \frac{2}{rf^2} \lt( \bar R \lt( 2\frac{\pl}{\pl w},\frac{\pl}{\pl\theta^a},\frac{\pl}{\pl v}, f^2 \na^b v \frac{\pl}{\pl\theta^b} \rt) + \bar R \lt( 2\frac{\pl}{\pl w},\frac{\pl v}{\pl\theta^a}\frac{\pl}{\pl v}, \frac{\pl}{\pl v}, 2\frac{\pl}{\pl w} \rt) \rt)\\
&= \frac{2}{rf^2} \lt( 2\frac{f^3}{r}f' - 2 f^2 (ff'' + (f')^2) \rt)\frac{\pl r}{\pl\theta^a}.
\end{align*}
\end{proof}

\end{document}